\documentclass[a4paper, 11pt]{article}

\usepackage[
            includefoot,  
            marginparwidth=0in,     
            marginparsep=0in,       
            margin=1.45in,               
            includemp]{geometry}

\usepackage{bbm}
\usepackage{float}
\usepackage{graphicx}                  
\usepackage{amssymb}
\usepackage{amsfonts}
\RequirePackage{amsmath}
\RequirePackage{amsthm}

\usepackage{color}

\usepackage{fancyhdr}

\newtheorem{thm}{Theorem}[section]

\newtheorem{prop}[thm]{Proposition}

\newtheorem{rmq}{Remark}[section]

\DeclareMathOperator{\Hess}{Hess}

\DeclareMathOperator{\Ent}{Ent}

\newcommand{\R}{\mathbb{R}}

\begin{document}

\title{A sharp symmetrized form of Talagrand's transport-entropy inequality for the Gaussian measure}
\date{\today}
\author{Max Fathi}

\maketitle

\begin{abstract}
This note presents a sharp transport-entropy inequality that improves on Talagrand's inequality for the Gaussian measure, arising as a dual formulation of the functional Santal\'o inequality. We also discuss some extensions and connections with concentration of measure.
\end{abstract}

\section{Introduction and main result}

Talagrand's inequality for the standard Gaussian measure $\gamma$ on $\R^d$, originally proved in \cite{Ta96}, states that for any probability measure $\nu$, we have
$$W_2(\nu, \gamma)^2 \leq 2\Ent_{\gamma}(\nu).$$

Here, $W_2$ stands for the usual $L^2$ Kantorovitch-Wasserstein distance on the space of probability measures, and $\Ent_{\gamma}$ stands for the relative entropy with respect to the Gaussian measure. 

The main result of this note is the following improvement: 

\begin{thm} \label{main_thm}
Let $\mu$ be a centered probability measure, and $\nu$ be another probability measure. Then 
$$W_2(\mu, \nu)^2 \leq 2\Ent_{\gamma}(\mu) + 2\Ent_{\gamma}(\nu).$$

Moreover, equality holds iff there exists a symmetric positive matrix $A$ such that $\mu$ is a non-degenerate centered Gaussian measure with covariance $A$ and $\nu$ is a Gaussian measure with covariance $A^{-1}$ (not necessarily centered). 
\end{thm}

We can then recover Talagrand's inequality by taking $\mu = \gamma$. The assumption that one of the two measures must be centered cannot be removed in general: if we take $\mu$ and $\nu$ both as non-centered standard Gaussian measures, with respective barycenter $m_1$ and $m_2$, the inequality would become $|m_1 - m_2|^2 \leq |m_1|^2 + |m_2|^2$, which does not hold for all choices of $m_1$ and $m_2$.

\begin{rmq}
It is easy to check using the triangle inequality for $W_2$ that Talagrand's inequality implies $W_2(\mu, \nu)^2 \leq 4\Ent_{\gamma}(\mu) + 4\Ent_{\gamma}(\nu)$. The point here is that we can improve the prefactor in such a way that the statement becomes strictly stronger than Talagrand's inequality. This may be useful for applications where sharp estimates are desirable. For example, we will not lose a factor $2$ when deriving concentration inequalities from the functional inequality. See also \cite{GRST} for applications of symmetrized transport-entropy inequalities to concentration of measure. 
\end{rmq}

It turns out that this inequality is dual to the functional Santal\'o inequality \cite{AKM, Leh09}. This connection was pointed out to us by N. Gozlan. 

\begin{thm}[Functional Santal\'o inequality] \label{thm_equ}
Let $f$ and $g$ be measurable functions on $\R^d$ satisfying $f(x) + g(y) \leq -x \cdot y$ for all $x, y \in \R^d$. If $e^f$ (or $e^g$) has its barycenter at zero, then 
\begin{equation} \label{funct_san}
\left(\int{e^fdx}\right)\left(\int{e^gdx}\right) \leq (2\pi)^d.
\end{equation}
\end{thm}

This statement is due to Lehec \cite{Leh09}, and improves on previous results of Ball \cite{Bal} (for even functions) and Artstein, Klartag and Milman \cite{AKM} (where the function with barycenter at zero was assumed to be concave). It is a functional generalization of a result of Santal\'o on volumes of convex bodies \cite{San49}. See also \cite{FM08} for related results. 

\begin{prop}
The sharp symmetrized Talagrand inequality and the functional Santal\'o inequality are equivalent. 
\end{prop}

Duality between transport-entropy inequalities and integral bounds go back to \cite{BG}, which gave a dual formulation of classical transport-entropy inequalities. 

The simplest way to prove Theorem \ref{main_thm} is to derive it from the above inequality by a duality argument. We shall nonetheless give an alternate proof, which still relies on results derived using the functional Santal\'o inequality. Despite not being the simplest proof, we believe it is of some interest, as it highlights the connection to optimal transport and calculus of variations. 

As was pointed out by Klartag \cite{Kla07} (and discovered independently by Barthe and Cordero-Erausquin), such inequalities can be pushed to uniformly log-concave measures using the Caffarelli contraction theorem, leading to the following variant of Theorem \ref{main_thm}. This statement was pointed out to us by Dario Cordero-Ersausquin.  

\begin{thm}
\label{extended_thm_ulc}
Let $\theta =e^{-V}dx$ be a symmetric, uniformly log-concave probability measure, that is the potential $V$ is smooth and satisfies $\Hess V \geq \alpha \operatorname{Id}$ for some $\alpha > 0$. Then for any symmetric probability measure $\mu$ and any other probability measure $\nu$, we have
$$W_2(\mu, \nu)^2 \leq \frac{2}{\alpha}(\Ent_{\theta}(\mu) + \Ent_{\theta}(\nu)).$$
\end{thm}

The symmetry assumptions on $\mu$ and $\theta$ could be relaxed, and we would then have to assume instead that $\int{Td\mu} = 0$, where $T$ is the optimal transport map sending $\theta$ onto $\gamma$. 

Section 2 will contain the proofs of the results we just described. Section 3 will present a reverse form of the improved Talagrand inequality, under some convexity and symmetry assumptions on the measures. Finally, Section 4 will describe a concentration estimate that will be easily deduced from Theorem \ref{main_thm}, in connection with Maurey's property $(\tau)$. 

\section{Proofs}

\subsection{Proof of Theorem \ref{main_thm}}

The proof uses the following two statements. The first is a result of Santambrogio \cite{San16}

\begin{thm} \label{thm_santambrogio}
Let $\mu$ be a centered probability measure that is not supported on a hyperplane. Then there exists an essentially continuous convex function $\varphi$, unique up to translations, such that $\rho = e^{-\varphi}dx$ is a probability measure on $\R^d$ whose pushforward by the map $\nabla \varphi$ is $\mu$. Moreover, it satisfies
$$\rho = \operatorname{argmin} \left\{ -\frac{1}{2}W_2(\mu, \nu)^2 + \Ent_{\gamma}(\nu); \nu \in \mathcal{P}_1(\R^d) \right\}.$$
\end{thm}

The first part of this statement was first proved by Cordero-Erausquin and Klartag \cite{CEK15}, and is the main focus of \cite{San16}. The second part is a byproduct of the method Santambrogio used, but turns out to be useful for our purpose. Note that this result contains the Talagrand inequality, which is obtained when taking $\mu$ to be the Gaussian measure. We refer to \cite{CEK15} for a definition of essential continuity. 

When $\rho$ is the minimizer, the above quantity can be rewritten as 
$$-\frac{1}{2}\int{|\nabla \varphi - x|^2e^{-\varphi}dx} + \Ent_{\gamma}(e^{-\varphi}).$$
The first term is the negative of the Fisher information of $\rho$, relative to the Gaussian, so this quantity can be identified as the negative of the deficit in the classical Gaussian logarithmic Sobolev inequality for $\rho$. This is a first hint that this result may be useful to study improvements to functional inequalities. Connections between deficit estimates for Gaussian functional inequalities and moment maps have also been investigated in \cite{KK18}

The second tool we shall use is a reverse form of the Gaussian logarithmic Sobolev inequality, proven in \cite{AKSW}, under some regularity assumptions. In \cite{CFGLSW}, an alternative, simpler proof was established, which removed those extra regularity assumptions. We point out that the simpler proof of \cite{CFGLSW} uses the functional Santal\'o inequality of \cite{AKM}. 

\begin{thm} \label{thm_acglsw}
Take $\rho = e^{-\varphi}dx$ a probability measure, and assume it is log-concave. Let $S(\rho) := -\int{ \log f d\rho}$ where $f= e^{-\varphi}$ with respect to the Lebesgue measure. Then 
$$S(\gamma) - S(\rho) \geq \frac{1}{2}\int{\log \det \nabla^2 \varphi d\rho}.$$
\end{thm}

We can now give the proof of Theorem \ref{main_thm}. We can assume without loss of generality that $\mu$ is absolutely continuous with respect to the Lebesgue measure. Let $f$ be the density of $\mu$ with respect to the Lebesgue measure, and consider the convex function $\varphi$ given by Theorem \ref{thm_santambrogio}. It satisfies almost everywhere the Monge-Amp\`ere PDE
$$e^{-\varphi} = f(\nabla \varphi)\det \nabla^2\varphi.$$ 
Taking the logarithm of this PDE, a computation shows that
$$S(\mu) = S(\rho) + \int{\log \det \nabla^2 \varphi d\rho}.$$

Applying Theorem \ref{thm_acglsw}, this implies 
$$2S(\gamma) - S(\mu) - S(\rho) \geq 0$$
which is equivalent to
\begin{align*}
\Ent_{\gamma}(\mu) &+ \Ent_{\gamma}(\rho) \geq \frac{1}{2}\int{|x|^2d\rho} + \frac{1}{2}\int{|x|^2d\mu} - d \\
&\geq \frac{1}{2}\int{|x|^2d\rho} + \frac{1}{2}\int{|\nabla \varphi|^2d\rho} - \int{x \cdot \nabla \varphi d\rho} \\
&= \frac{1}{2}\int{|x - \nabla \varphi|^2d\rho} = \frac{1}{2}W_2(\mu, \rho)^2. 
\end{align*}

Here we have used the inequality $\int{x \cdot \nabla \varphi d\rho} \leq d$. If $\varphi$ was smooth, this would be an equality, immediately justified by an integration by parts. Since $\varphi$ is not necessarily very smooth, we cannot do this, but \cite{CEK15} justified that when $\varphi$ is essentially continuous (which they proved is the case here), the inequality is still true despite the lack of regularity.

Equivalently, 
$$\Ent_{\gamma}(\rho) - \frac{1}{2}W_2(\rho, \mu)^2 \geq - \Ent_{\gamma}(\mu).$$

But since Theorem \ref{thm_santambrogio} states that $\rho$ is a minimizer of $\nu \longrightarrow \Ent_{\gamma}(\nu) - \frac{1}{2}W_2(\nu, \mu)^2$, this implies that for any probability measure $\nu$ with finite first moment
$$\Ent_{\gamma}(\nu) - \frac{1}{2}W_2(\nu, \mu)^2 \geq - \Ent_{\gamma}(\mu)$$
which after rearranging the terms is the statement we were aiming to prove. 

Moreover, for equality to hold, it must also hold in the inequality $S(\gamma) - S(\rho) \geq \frac{1}{2}\int{\log \det \nabla^2 \varphi d\rho}$. Since cases of equality here are known to only be when $\rho$ is Gaussian with some positive definite covariance matrix $A$, and $\mu$ is a pushforward of $\rho$ by $\nabla \varphi$, it is then also Gaussian, and its covariance matrix is $A^{-1}$.  A standard computation confirms that equality indeed holds in such a situation. 

\subsection{Proof of Proposition \ref{thm_equ}} 

First, for convenience we reformulate the improved Talagrand inequality as 

\begin{equation} \label{reformulated_sym_t2}
\underset{\pi}{\inf} \hspace{1mm} -\int{x \cdot y d\pi} \leq \Ent_{dx}(\mu) + \Ent_{dx}(\nu) + d\log(2\pi)
\end{equation}

The left-hand side of this inequality is still a transport cost, but with cost $-x \cdot y$ instead of $|x-y|^2$. 

The equivalence between the improved Talagrand inequality and the functional Santal\'o inequality is a consequence of the dual formulations of transport cost and entropy: 

\begin{equation} \label{kant_dual} 
\underset{\pi}{\inf} \hspace{1mm} -\int{x \cdot y d\pi} = \underset{f(x) + g(y) \leq -x\cdot y}{\sup} \int{fd\mu} + \int{gd\nu};
\end{equation}

\begin{equation}
\Ent_{dx}(\mu) = \underset{f}{\sup} \int{fd\mu} - \log \int{e^fdx}.
\end{equation}

The first identity is the Kantorovitch dual formulation of the optimal transport problem with cost $-x\cdot y$ (see for example \cite{Vil05}), while the second identity is the classical reformulation of entropy as the Legendre transform of the log-Laplace functional. 

Let us first prove that the symmetrized Talagrand inequality implies the functional Santal\'o inequality. Take $f$ and $g$ such that $f(x) + g(y) \leq -x\cdot y$ for all $x, y$, and such that $\int{xe^fdx} = 0$. Applying  the Talagrand inequality with $\mu = e^f\left(\int{e^f}\right)^{-1}$ and $\nu = e^g\left(\int{e^g}\right)^{-1}$, we get after taking into account \eqref{kant_dual} 
\begin{equation} \label{eq_sant}
\int{fd\mu} + \int{gd\nu} \leq \int{fd\mu} - \log \int{e^fdx} + \int{gd\nu} - \log \int{e^gdx} + d\log(2\pi)
\end{equation}
which is easily seen to be the same thing as \eqref{funct_san}. 

For the converse, fix $\mu$ a centered probability measure, and $\nu$ any probability measure, and consider $f$ and $g$ satisfying $f(x) + g(y) \leq - x\cdot y$. There exists $\lambda \in \R^d$ such that $\int{xe^{f + \lambda\cdot x}dx} = 0$. Indeed, the condition on $f$ and $g$ implies $f$ decays to $-\infty$ at infinity faster than any linear function, so this quantity is well defined for all $\lambda$, it is a smooth monotone function in $\lambda$ so its range is convex, and any coordinate is unbounded, so its range is the whole space. Let $\tilde{f}(x) = f(x) + \lambda \cdot x$ and $\tilde{g}(y) = g(y + \lambda)$. Then $\tilde{f}(x) + \tilde{g}(y) \leq -x \cdot y$ and $\int{xe^{\tilde{f}}dx} = 0$. Applying the Santal\'o inequality, we get
$$\log \int{e^{\tilde{f}}dx} + \log \int{e^{\tilde{g}}dx} \leq d\log(2\pi).$$
Since $\int{fd\mu} = \int{\tilde{f}d\mu}$ because $\mu$ is centered, and since $\int{e^{g}dx} = \int{e^{\tilde{g}}dx}$, we get 
$$\int{fd\mu} + \int{gd\nu} \leq \int{\tilde{f}d\mu} - \log\int{e^{\tilde{f}}dx} + \int{gd\nu} - \log \int{e^{g}dx} + d\log(2\pi)$$
$$\leq \Ent_{dx}(\mu) + \Ent_{dx}(\nu) +d\log(2\pi)$$
and taking the supremum over all $f$ and $g$ yields \eqref{reformulated_sym_t2}, which concludes the proof. 

\subsection{Proof of Theorem \ref{extended_thm_ulc}}

The argument is the same as in \cite{Kla07}, we use the Caffarelli contraction theorem \cite{Caf00}, which states that there exists a $\alpha^{-1/2}$ Lipschitz map $T$ sending $\gamma$ onto $\theta$. Without loss of generality, we may assume $\theta$ has a density. Indeed, otherwise its support is included in a hyperplane, and we can iterate the argument in $\R^{d-1}$, unless it is supported on a single point, in which case there is nothing to prove. Then the map $T$ is invertible, and we can consider its inverse $T^{-1}$ (which is the optimal transport map sending $\theta$ onto $\gamma$). Let $\tilde{\mu}$ (resp. $\tilde{\nu}$) be the image of $\mu$ (resp. $\nu$) by $T^{-1}$. By symmetry of $\theta$, $T$ is a symmetric function, and hence $\tilde{\mu}$ is still a centered measure. We then have
\begin{align*}
W_2(\mu, \nu)^2 &\leq \frac{1}{\alpha}W_2(\tilde{\mu}, \tilde{\nu})^2 \leq \frac{2}{\alpha}(\Ent_{\gamma}(\tilde{\mu}) + \Ent_{\gamma}(\tilde{\nu})) \\
&= \frac{2}{\alpha}(\Ent_{\gamma \circ T}(\tilde{\mu} \circ T) + \Ent_{\gamma\circ T}(\tilde{\nu}\circ T)) = \frac{2}{\alpha}(\Ent_{\theta}(\mu) + \Ent_{\theta}(\nu))
\end{align*}
which completes the proof. We could remove the symmetry assumption on $\theta$ and $\mu$ by requiring instead that the image of $\mu$ by $T^{-1}$ is centered. 

\subsection{A remark on stability}

Given a functional inequality $\mathcal{F}(f) \leq \mathcal{G}(f)$ with sharp constants for which all equality cases are known, a natural question is to determine whether one can prove an improvement of the form $\mathcal{F}(f) + d(f, E)^{\alpha} \leq \mathcal{G}(f)$, where $E$ is the set of functions for which equality holds, and $d$ is some suitable distance on the space of functions or measures considered. This problem has been recently studied for several Gaussian functional inequalities, such as the isoperimetric problem \cite{Eld, BBJ}, the logarithmic Sobolev inequality \cite{Cou, FIL16, BGG} and Talagrand's inequality \cite{FIL16, CE}. In particular, \cite{KK18} investigated applications of the moment map problem to such deficit estimates.  

In \cite{CW17}, a deficit estimate for the inverse logarithmic Sobolev inequality was established. It takes the following form: 

\begin{thm} [Caglar and Werner 2017]
Let $\mu = e^{-f}dx$ be a log-concave probability measure on $\R^d$. For any $\epsilon \in (0, \epsilon_0)$, if $\int{\log \det (\Hess f)d\mu} \geq 2(S(\gamma) - S(\mu)) - \epsilon$, then there exists $c > 0$ and a positive definite matrix $A$ such that
$$\int_{B(0, R(\epsilon))}{\left|\frac{|x|^2}{2} - c - f(Ax)\right|dx} \leq \eta\epsilon^{\frac{1}{129d^2}}$$
where $\epsilon_0, R, \eta$ depend on $d$ and $R(\epsilon) \longrightarrow +\infty$ as $\epsilon$ goes to zero. 
\end{thm}

This result was established using a stability estimate for the functional Santal\'o inequality that was obtained in \cite{BBF14}. In high-dimensional situations, this is not a very good estimate, and it is an open problem whether the exponent $1/(129d^2)$ can be replaced by a constant independent of the dimension. 

Using this estimate to refine the proof of Theorem \ref{main_thm}, we straightforwardly obtain

\begin{thm}
Let $\mu$ be a centered measure, and assume that $\nu$ is a probability measure such that
$$\Ent_{\gamma}(\mu) + \Ent_{\gamma}(\nu) \leq \frac{1}{2}W_2(\mu, \nu)^2 + \epsilon$$
for some $\epsilon < \epsilon_0$. Then there exists $c > 0$, a positive definite matrix $A$ and a point $x_0 \in \R^d$ such that the moment map $\varphi$ of $\mu$ satisfies
$$\int_{B(0, R(\epsilon))}{\left|\frac{|x|^2}{2} - c - \varphi(Ax + x_0)\right|dx} \leq \eta\epsilon^{\frac{1}{129d^2}}$$
where $\epsilon_0, R, \eta$ depend on $d$ and $R(\epsilon) \longrightarrow +\infty$ as $\epsilon$ goes to zero. 
\end{thm}

\section{A reverse inequality}

We shall use a reverse form of the Santal\'o inequality for unconditional measures to derive a reverse form of our transport-entropy inequality. It is not clear to us if this inequality has any application, but since reverse Santal\'o inequalities have attracted some attention, in relation to the Mahler conjecture, we felt this problem was natural, and the estimate worth writing down. 

A function or probability measure is said to be unconditional if it is invariant by all symmetries with respect to a hyperplane $\{x_i = 0\}$ for any $i \in \{1,..,n\}$. In \cite{KM05} (see also \cite{FM08, AS}), the following inverse functional Santal\'o inequality was established: 

\begin{thm}[Klartag and Milman 2005]
Let $f$ be an unconditional convex function. Then 
$$\log \int{e^{-f}dx} + \log \int{e^{-f^*}dx} \geq d\log 4. $$
\end{thm}

Actually, their result only requires $f$ to reach its minimum at the origin, but it is this weaker version we shall use. 

\begin{thm}
Let $\mu = e^{-f}dx$ be an unconditional log-concave measure, and let $\mu^* = e^{-f^*}dx$. Then 
$$\Ent_{\gamma}(\mu) + \Ent_{\gamma}(\mu^*) \leq \frac{1}{2}W_2(\mu, \mu^*)^2 + \frac{d}{2}\log(\pi/2).$$
\end{thm}

\begin{proof}
Let $\mu$ and $\nu$ be two unconditional log-concave measures, and denote by $\mathcal{F}_{uc}$ the set of all unconditional convex functions. We can localize the duality formulas
$$\underset{\pi}{\inf} \hspace{1mm} \int{-x\cdot y d\pi} = \underset{f \in \mathcal{F}_{uc}}{\sup} \hspace{1mm} \int{fd\mu} + \int{f^*d\nu};$$
$$\Ent_{dx}(\mu) + \Ent_{dx}(\mu^*) = \underset{f \in \mathcal{F}_{uc}}{\sup} \hspace{1mm} -\int{fd\mu} - \int{f^*d\mu^*} - \log\int{e^{-f}dx} - \log\int{e^{-f^*}dx}.$$
For the second formula, this is trivial since we know the optimizer is the logarithm of the density. For the first one, this is a consequence of the convexity of the optimizer in Kantorovitch duality, and that the optimizer necessarily inherits symmetry properties shared by both measures. 

Using these formulas and following the same approach as in the proof of Theorem \ref{main_thm}, we get
$$\Ent_{dx}(\mu) + \Ent_{dx}(\mu^*) \leq \underset{\pi}{\inf} \hspace{1mm} \int{-x\cdot y d\pi} -d\log(4).$$ 
Adding second moments and a constant, we get
$$\Ent_{\gamma}(\mu) + \Ent_{\gamma}(\mu^*) \leq \frac{1}{2}W_2(\mu, \mu^*)^2 + \frac{d}{2}\log(\pi/2).$$
\end{proof}

\section{A concentration of measure estimate for the Gaussian distribution}

\begin{thm}
Let $A$ be a measurable set with $\gamma(A) > 0$ and $\int_A{xd\gamma} = 0$. Let $A_r := \{x; d(x, A) \leq r\}$. We have
$$1 - \gamma(A_r) \leq \gamma(A)^{-1}e^{-r^2/2}.$$
\end{thm}

It is known that this kind of estimate can be obtained as a consequence of the Santal\'o inequality via property $(\tau)$. If we use the classical property $(\tau)$ of Maurey \cite{Mau91}, we obtain this estimate with constant $1/4$ instead of $1/2$ in the exponent, but for general sets. The case of even sets with constant $1/2$ is an immediate consequence of the symmetric property $(\tau)$ obtained by Lehec \cite{Leh08}. See also the survey \cite{GL} for the relationship between property $(\tau)$ and transport inequalities. The point here is that deducing this concentration bound from the improved Talagrand inequality is completely straightforward. 

\begin{proof}
The proof follows Marton's classical argument for deducing concentration inequalities from transport-entropy inequalities \cite{Ma97}. Without loss of generality, we can assume $\gamma(A_r) < 1$. Let $\mu := \gamma(A)^{-1}\mathbbm{1}_Ad\gamma$ and $\nu :=  \gamma(A_r^c)^{-1}\mathbbm{1}_{A_r^c}d\gamma$. We have
\begin{equation*}
r^2 \leq W_2(\mu, \nu)^2 \leq 2\Ent_{\gamma}(\mu) + 2\Ent_{\gamma}(\nu) = -2\log \gamma(A) - 2\log(1 - \gamma(A_r))
\end{equation*}
which implies
$$1 - \gamma(A_r) \leq \gamma(A)^{-1}e^{-r^2/2}.$$
\end{proof}

\textbf{\underline{Acknowledgments}} : This material is partly based upon work supported by the National Science Foundation under Grant No. 1440140, while the author was in residence at the Mathematical Sciences Research Institute in Berkeley, California, during the Fall semester 2017. This work was also partially supported by the EFI project ANR-17-CE40-0030 of the French National Research Agency (ANR), the France-Berkeley Fund via the project `Geometric and Functional Inequalities in Probability and Information Theory' and by ANR-11-LABX-0040-CIMI within the program ANR-11-IDEX-0002-02. I would like to thank Nathael Gozlan for pointing out to me the connection with the functional Santal\'o inequality, and Dario Cordero-Erausquin for pointing out to me Theorem \ref{extended_thm_ulc}.  I also thank Franck Barthe, Thomas Courtade and Michel Ledoux for discussions on this topic, and Micha\l \hspace{1mm} Sztrelecki for explaining to me the $(\tau)$ property.

\end{document}